\documentclass[11pt,reqno]{amsart}

\usepackage{amssymb}
\usepackage{amsmath,amsthm,amsfonts,amssymb,latexsym,mathrsfs,color,hyperref}
\usepackage{graphicx}
\usepackage{xspace}
\usepackage{multirow}
\usepackage{diagbox}
\usepackage{capt-of}
\usepackage{tikz}
\tikzset{
	level 1/.style = {sibling distance = 1.5cm},
	level 2/.style = {sibling distance = 0.8cm},
    level distance = 0.9 cm
}
\usetikzlibrary {decorations.pathmorphing}
\usetikzlibrary{matrix, positioning}
\tikzstyle{snakeline} = [decorate, decoration={snake, amplitude=.4mm, segment length=2mm}]

\usepackage{tikz-cd}
\usepackage[boxsize=1.3em, centerboxes]{ytableau}

\usepackage{tikz-qtree}
\tikzset{every tree node/.style={minimum width=0.1cm,draw,circle},
         blank/.style={draw=none},
         edge from parent/.style=
         {draw,edge from parent path={(\tikzparentnode) -- (\tikzchildnode)}},
         level distance=0.8cm}

\newtheorem{theorem}{Theorem}
\newtheorem{corollary}[theorem]{Corollary}
\newtheorem{proposition}[theorem]{Proposition}

\newtheorem{lemma}[theorem]{Lemma}
\newtheorem{definition}[theorem]{Definition}
\newtheorem{example}[theorem]{Example}
\newtheorem{problem}[theorem]{Problem}

\newtheorem*{LSp}{Label Schema}

\newcommand{\YWCT}{{\rm YWCT\,}}
\newcommand{\SYT}{{\rm SYT\,}}

\newcommand{\even}{{\rm even\,}}
\newcommand{\negg}{{\rm neg\,}}

\newcommand{\redu}{{\rm red\,}}

\newcommand{\plat}{{\rm plat\,}}
\newcommand{\ap}{{\rm ap\,}}
\newcommand{\lap}{{\rm lap\,}}

\newcommand{\single}{{\rm single\,}}

\newcommand{\des}{{\rm des\,}}

\newcommand{\exc}{{\rm exc\,}}

\newcommand{\aexc}{{\rm aexc\,}}

\newcommand{\cyc}{{\rm cyc\,}}

\newcommand{\fix}{{\rm fix\,}}

\newcommand{\msn}{\mathfrak{S}_n}

\newcommand{\ms}{\mathfrak{S}}

\newcommand{\mq}{\mathcal{Q}}
\newcommand{\mqn}{\mathcal{Q}_n}

\newcommand{\asc}{{\rm asc\,}}

\newcommand{\Eulerian}[2]{\genfrac{<}{>}{0pt}{}{#1}{#2}}
\newcommand{\Stirling}[2]{\genfrac{\{}{\}}{0pt}{}{#1}{#2}}
\newcommand{\stirling}[2]{\genfrac{[}{]}{0pt}{}{#1}{#2}}
\newcommand{\arxiv}[1]{\href{http://arxiv.org/abs/#1}{\texttt{arXiv:#1}}}
\title{Eulerian-type polynomials over Stirling permutations and box sorting algorithm}
\author[S.-M.~Ma]{Shi-Mei Ma}
\address{School of Mathematics and Statistics, Shandong University of Technology, Zibo 255000, Shandong, P.R. China}
\email{shimeimapapers@163.com (S.-M. Ma)}
\author[J.-Y.~Liu]{Jun-Ying Liu}
\address{School of Mathematics and Statistics, Shandong University of Technology, Zibo, Shandong 255000, P.R. China}
\email{jyliu6@163.com (J.-Y. Liu)}
\author{Jean Yeh}
\address{Department of Mathematics, National Kaohsiung Normal University, Kaohsiung 82444, Taiwan}
\email{chunchenyeh@nknu.edu.tw (J. Yeh)}
\author[Y.-N. Yeh]{Yeong-Nan Yeh}
\address{College of Mathematics and Physics, Wenzhou University, Wenzhou 325035, P.R. China}
\email{mayeh@math.sinica.edu.tw (Y.-N. Yeh)}
\subjclass[2010]{Primary 05A19; Secondary 05E05}
\begin{document}

\maketitle
\begin{abstract}
It is well known that ascents, descents and plateaux are equidistributed over the set of classical Stirling permutations. 
Their common enumerative polynomials are the second-order Eulerian polynomials, which have been extensively studied by many researchers.
This paper is divided into three parts. The first parts gives a convolution formula 
for the second-order Eulerian polynomials, which simplifies a result of Gessel.  
As an application, a determinantal expression for the second-order Eulerian polynomials is obtained. 
We then investigate the convolution formula of the trivariate second-order Eulerian polynomials. 
Among other things, by introducing three new statistics: proper ascent-plateau, 
improper ascent-plateau and trace, we discover that a six-variable Eulerian-type polynomial over a class of restricted Stirling permutations equals a 
six-variable Eulerian-type polynomial over signed permutations. 
By special parametrizations, we make use of Stirling permutations to give a unified interpretations of the $(p,q)$-Eulerian polynomials 
and derangement polynomials of types $A$ and $B$. The third part presents a box sorting algorithm which leads to a bijection between
the terms in the expansion of $(cD)^nc$ and ordered weak set partitions,
where $c$ is a smooth function in the indeterminate $x$ and $D$ is the derivative with respect to $x$.
Using a map from ordered weak set partitions to standard Young tableaux, we find an expansion of
$(cD)^nc$ in terms of standard Young tableaux. Combining this with grammars, we provide three interpretations of the second-order Eulerian polynomials.

\bigskip

\noindent{\sl Keywords}: Stirling permutations; Signed permutations; Proper ascent-plateaux; Improper ascent-plateau; Box sorting algorithm; Standard Young tableaux
\end{abstract}
\date{\today}
\tableofcontents
\section{Introduction}
Many polynomials can be generated by successive differentiations of a given base function, see~\cite[Section~4.3]{Hwang20} for a survey.
Here we provide two such polynomials.
The Eulerian polynomials $A_n(x)$ first appearance in the following series summation or successive differentiation:
\begin{equation}\label{Anx-poly-def}
\sum_{k=0}^\infty k^nx^k=\left(x\frac{\mathrm{d}}{\mathrm{d}x}\right)^n\frac{1}{1-x}=\frac{A_n(x)}{(1-x)^{n+1}}.
\end{equation}
The {\it second-order Eulerian polynomials} $C_n(x)$ can be defined by
\begin{equation}\label{Cnx-poly-def}
\sum_{k=0}^\infty \Stirling{n+k}{k}x^k=\left(\frac{x}{1-x}\frac{\mathrm{d}}{\mathrm{d}x}\right)^n\frac{x}{1-x}=\frac{C_n(x)}{(1-x)^{2n+1}},
\end{equation}
where $\Stirling{n}{k}$ is the {\it Stirling numbers of the second kind}, i.e., the number of partitions of $[n]:=\{1,2,\ldots,n\}$ into exactly $k$ blocks, see~\cite{Buckholtz,Carlitz65}.
The polynomials $A_n(x)$ and $C_n(x)$ share several similar properties, including recursions~\cite{Gessel78}, real-rootedness~\cite{Liu07} and combinatorial expansions~\cite{Chen22}.
For example, we have
\begin{equation*}\label{Eulerian01}
A_{n}(x)=nxA_{n-1}(x)+x(1-x)\frac{\mathrm{d}}{\mathrm{d}x}A_{n-1}(x),~A_0(x)=1;
\end{equation*}
\begin{equation}\label{Cnx-recu}
C_{n+1}(x)=(2n+1)xC_n(x)+x(1-x)\frac{\mathrm{d}}{\mathrm{d}x}C_n(x),~C_0(x)=1.
\end{equation}

This paper is motivated by the following problem.
\begin{problem}\label{problem0}
 If there exists a function $f_n(x_1,x_2,\ldots,x_n)$ such that 
 $$C_n(x)=f_n\left(A_1(x),A_2(x),\ldots,A_n(x)\right).$$
\end{problem}

For $\mathbf{m}=(m_1,m_2,\ldots,m_n)\in \mathbb{N}^n$, let $\mathbf{n}=\{1^{m_1},2^{m_2},\ldots,n^{m_n}\}$ be a multiset,
where the element $i$ appears $m_i$ times. A {\it multipermutation} of $\mathbf{n}$ is a sequence of its elements.
For any word over $\mathbf{n}$, we say that the {\it reduced
form} of $w$, written $\redu(w)$, is equal to the word obtained by replacing each of the
occurrences of the $i$th smallest number in $w$ with the number $i$. 
Denote by $\ms_{\mathbf{n}}$ the set of multipermutations of $\mathbf{n}$.
We say that a multipermutation $\sigma$ of $\mathbf{n}$ is a {\it Stirling permutation} if for each $i$, $1\leqslant i\leqslant n$, 
all letters occurring between the two occurrences of $i$ are at least $i$.
Denote by $\mq_\mathbf{n}$ the set of Stirling permutations of $\mathbf{n}$.
When $m_1=\cdots=m_n=1$, the set $\mq_\mathbf{n}$ reduces to the symmetric group $\msn$. When $m_1=m_2=\cdots=m_n=2$,
the set $\mq_\mathbf{n}$ reduces to the set $\mq_n$ (the set of classical Stirling permutations), which is defined by Gessel-Stanley~\cite{Gessel78} when they considering~\eqref{Cnx-poly-def}.
Except where explicitly stated, we always assume that all Stirling permutations belong to $\mqn$.
For example, $\mq_1=\{11\},~\mq_2=\{1122,1221,2211\}$.

For $\sigma\in\mq_{\mathbf{n}}$, any entry $\sigma_{i}$  is called an {\it ascent} (resp.~{\it descent},~{\it plateau}) if $\sigma_{i}<\sigma_{i+1}$ (resp.~$\sigma_{i}>\sigma_{i+1}$,~$\sigma_{i}=\sigma_{i+1}$), where $i\in \{0,1,2,\ldots,m_1+m_2+\cdots+m_n\}$ and we set $\sigma_0=\sigma_{m_1+m_2+\cdots+m_n+1}=0$.
Let $\asc(\sigma)$ (resp.~$\des(\sigma)$,~$\plat(\sigma)$) be the number of ascents, descents and plateaux of $\sigma$.
The {\it Eulerian polynomials} can also be defined by $$A_n(x)=\sum_{\pi\in\msn}x^{\des(\pi)}=\sum_{\pi\in\msn}x^{\asc(\pi)}=\sum_{k=0}^n\Eulerian{n}{k}x^k,$$
where $\Eulerian{n}{k}$ are known as the {\it Eulerian numbers}.
When $\sigma\in\mqn$, we always set $\sigma_0=\sigma_{2n+1}=0$. It is well known that $A_n(x)$ is symmetric, see~\cite{Ma19} for instance.
Gessel-Stanley~\cite{Gessel78} found that
$$C_n(x)=\sum_{\sigma\in\mqn}x^{\des(\sigma)}=\sum_{\sigma\in\mqn}x^{\asc(\sigma)}.$$
The trivariate second-order Eulerian polynomials are defined by
\begin{equation*}\label{Cnxyz01}
C_n(x,y,z)=\sum_{\sigma\in\mqn}x^{\asc{(\sigma)}}y^{\des(\sigma)}z^{\plat{(\sigma)}}.
\end{equation*}
The first few $C_n(x,y,z)$ are given as follows:
\begin{align*}
C_0(x,y,z)&=x,~C_1(x,y,z)=xyz,~C_2(x,y,z)=xyz(yz+xy+ xz),\\
C_3(x,y,z)&=xyz(x^2y^2+x^2z^2+y^2z^2+4xy^2z+4x^2yz+4xyz^2).
\end{align*}
Dumont~\cite[p.~317]{Dumont80} discovered that
\begin{equation}\label{Dumont80}
C_{n+1}(x,y,z)=xyz\left(\frac{\partial}{\partial x}+\frac{\partial}{\partial y}+\frac{\partial}{\partial z}\right)C_n(x,y,z),
\end{equation}
which implies that the polynomials $C_n(x,y,z)$ are symmetric in the variables $x,y$ and $z$. It was independently discovered by B\'ona~\cite{Bona08} that
$C_n(x)=\sum_{\sigma\in\mqn}x^{\plat(\sigma)}$.
In recent years, there has been considerable study on Stirling permutations and their variants, see~\cite{Brualdi20,Chen23,Elizalde,Haglund12,Janson11,Kuba21,Lin21,Liu23,Ma19,Ma23}.

Note that each nonempty Stirling permutation $\sigma\in\mqn$ can be represented as
$\sigma'1\sigma''1\sigma'''$, where $\sigma'$, $\sigma''$ and $\sigma'''$ are all Stirling permutations (may be empty).
Clearly, the descent number of $\sigma$ is the sum of the descent numbers of $\sigma'$, $\sigma''$ and $\sigma'''$, 
unless $\sigma'''$ is empty in which case $\sigma$ has an additional descent. Motivated by this observation,
Gessel~\cite{Gessel20} found that
$$\frac{\mathrm{d}}{\mathrm{d}z}C(x;z)=C^2(x;z)(C(x;z)+x-1),$$
where $C(x;z)=\sum_{n=0}^\infty C_n(x)\frac{z^n}{n!}$.
Extracting the coefficient of $\frac{z^n}{n!}$, one can easily deduce that
\begin{equation}\label{Cnx-convolution}
C_{n+1}(x)=x\sum_{k=0}^n\binom{n}{k}C_k(x)C_{n-k}(x)+\sum_{k=0}^{n-1}\binom{n}{k}\left(\sum_{j=0}^k\binom{k}{j}C_{j}(x)C_{k-j}(x)\right)C_{n-k}(x).
\end{equation}
Our point of departure is the following problem.
\begin{problem}\label{P1}
Whether there is a simplified version of~\eqref{Cnx-convolution}?
\end{problem}

A {\it partition} of $n$ is a weakly decreasing sequence of nonnegative integers:
$\lambda=(\lambda_1,\lambda_2,\ldots,\lambda_{\ell})$,
where $\sum_{i=1}^{\ell}\lambda_i=n$. Each $\lambda_i$ is called a {\it part} of $\lambda$.
If $\lambda$ is a partition of $n$, then we write $\lambda\vdash n$.
We denote by $m_i$ the number of parts equals $i$. By using the multiplicities, we also denote $\lambda$ by $(1^{m_1}2^{m_2}\cdots n^{m_n})$.
The {\it length} of $\lambda$, denoted $\ell(\lambda)$, is the maximum subscript $j$ such that $\lambda_j>0$.
The {\it Ferrers diagram} of $\lambda$ is graphical representation of $\lambda$ with $\lambda_i$ boxes in its $i$th row and the boxes are left-justified.
For a Ferrers diagram $\lambda\vdash n$ (we will often identify a partition with its Ferrers
diagram), a {\it standard Young tableau} ($\SYT$, for short) of shape $\lambda$ is a filling of the $n$ boxes of $\lambda$ with the integers $\{1,2,\ldots, n\}$ such that each of those
integers is used exactly once, and all rows and columns are increasing (from left to right, and from bottom to top, respectively),
and we number its rows starting from the bottom and going above.

Let $\operatorname{SYT}_\lambda$ be the set of standard Young tableaux of shape $\lambda$.
Set $\SYT(n)=\bigcup_{\lambda\vdash n}\operatorname{SYT}_\lambda$.
The {\it descent set} for $T\in\SYT(n)$ is defined by
$$\operatorname{Des}(T)=\{i\in[n-1]: i+1~{\text{appears in an upper row of $T$ than $i$}} \}.$$
Denote by $\#V$ the cardinality of a set $V$.
Let $\des(T)=\#\operatorname{Des}(T)$.
The Robinson-Schensted correspondence is a bijection from permutations to pairs of standard Young tableaux of the
same shape. This correspondence and its generalization, the
Robinson-Schensted-Knuth correspondence, have become centrepieces of enumerative and algebraic combinatorics due to their many applications and properties.
Let $f^{\lambda}=\#\operatorname{SYT}_\lambda$.
An application of the Robinson-Schensted correspondence is given as follows:
\begin{equation}\label{AnxSYT01}
A_n(x)=\sum_{\pi\in\msn}x^{\des(\pi)}=\sum_{\lambda\vdash n}f^{\lambda}\sum_{T\in\operatorname{SYT}_\lambda}x^{\des(T)+1},
\end{equation}
which has been extended to skew shapes, see~\cite{Adin19,Huang20} for instances. 

This paper is organized as follows.
In the next section, we present a result concerning Problem~\ref{P1}.
In Section~\ref{Section3}, we investigate the convolution formula for the trivariate second-order Eulerian polynomials $C_n(x,y,z)$.
A deep connection between signed permutations in the hyperoctahedral group and Stirling permutations is established.
In particular, we introduce three new statistics on Stirling permutations: proper ascent-plateau, 
improper ascent-plateau and trace. A special case of Corollary~\ref{corl2} says that
\begin{equation*}
\sum_{\sigma\in \mq_{n+1}^{(1)}}s^{\operatorname{impap}(\sigma)}t^{\operatorname{bk}_2(\sigma)}q^{\operatorname{tr}(\sigma)}
=2^n\sum_{\pi\in\msn}{\left(\frac{t+s}{2}\right)}^{\fix(\pi)}q^{\cyc(\pi)},
\end{equation*}
where ${\mq}_{n}^{(1)}$ is the set of Stirling permutations over the multiset $\{1,2,2,3,3,\ldots,n,n\}$.
In Section~\ref{Section4}, using standard Young tableaux, we provide certain combinatorial expansions of the Eulerian polynomials of types $A$ and $B$ as well as the second-order Eulerian polynomials. In particular, in Theorem~\ref{thm-Cnxy}, we give an affirmative answer to Problem~\ref{problem0}, i.e., 
$$C_n(x)=\sum_{T\in\operatorname{SYT}{(n)}}\prod_{i=1}^n\sigma_i(T){A_i(x)}^{w_i(T)},$$
where $w_i(T)$ is the number of rows in the standard Young tableau $T$ with $i$ elements and $\prod_{i=1}^n\sigma_i(T)$ 
is the $g$-index of $T$ introduced in~\cite{Han23}, see Table~\ref{tab:dummy-c3xy} for an illustration.
\section{Convolution formulas for the second-order Eulerian polynomials}
Since descents and plateaux are equidistributed over Stirling permutations in $\mqn$, it is natural to 
count Stirling permutation by the plateau statistics.
We now present a variant of~\eqref{Cnx-convolution}.
\begin{theorem}\label{lemmaCnx}
For any $n\geqslant 1$, one has
$$C_n(x)=nxC_{n-1}(x)+\sum_{r=1}^{n-1}\binom{n}{n-r+1}C_{n-r}(x)C_{r-1}(x).$$
When $x=1$, it reduces to
$$(2n-1)!!=n\left((2n-3)!!\right)+\sum_{r=1}^{n-1}\binom{n}{n-r+1}\left((2n-2r-1)!!\right)\left((2r-3)!!\right)$$
 \end{theorem}
\begin{proof} 
Given a Stirling permutation $\sigma$ in $\mqn$.
Consider a decomposition of $\sigma$: $i\sigma' i\sigma''$, where $\sigma'$ and $\sigma''$ are Stirling permutations (may be empty).
We distinguish two cases:
\begin{itemize}
  \item [$(i)$] If $\sigma$ begins with the plateau $ii$, i.e. $\sigma'=\emptyset$, then there are $n$ choices for $i$. When count the number of plateaux of $\sigma$, this case gives the term 
  $nxC_{n-1}(x)$.
  \item [$(ii)$] When  $\sigma'\neq\emptyset$, choose a multiset $$\{a_1,a_1,a_2,a_2,\ldots,a_{n-r+1},a_{n-r+1}\}$$ in $\binom{n}{n-r+1}$ ways, where $1\leqslant i=a_1<a_2<\cdots<a_{n-r+1}\leqslant n$. Let $S_{n-r}$ be the set of Stirling permutations over $\{a_2,a_2,a_3,a_3,\ldots,a_{n-r+1},a_{n-r+1}\}$, and let $\sigma'\in S_{n-r}$.
If we count the number of plateaux of $\sigma'$, then $S_{n-r}$ contributes to the term $C_{n-r}(x)$. Let $\overline{S}_{r-1}$ be the set of Stirling permutations over  $\{1,1,2,2,\ldots,n,n\}\setminus\{a_1,a_1,a_2,a_2,\ldots,a_{n-r+1},a_{n-r+1}\}$. 
 Then $\redu(\sigma'')\in\mq_{r-1}$ and $\overline{S}_{r-1}$ contributes to the term $C_{r-1}(x)$. 
\end{itemize}
The aforementioned two cases exhaust all the possibilities, and this completes the proof.
\end{proof}
 
Let $a_n=2^nC_n(1/2)$, which is the number of series-reduced rooted trees with $n+1$ labeled leaves~\cite[
A000311]{Sloane}. The first few $a_n$ are $1,4,26,236, 2752$. 
It follows from Theorem~\ref{lemmaCnx} that 
$$a_n=na_{n-1}+2\sum_{r=1}^{n-1}\binom{n}{n-r+1}a_{n-r}a_{r-1},~a_0=a_1=1.$$
Very recently, Bitonti-Deb-Sokal~\cite{Bitonti24} obtained a continued fraction expression of $a_n$.

\begin{lemma}[{\cite{Liu07}}]\label{Liu07}
Let $F(x)=a(x)f(x)+b(x)\frac{\mathrm{d}}{\mathrm{d}x}f(x)$, where $a(x)$ and $b(x)$ are two real polynomials. Both $f(x)$ and $F(x)$ have only positive leading coefficients, and $\deg F=\deg f$ or $\deg f+1$. If $f(x)$ has only real zeros and $b(r)\leqslant 0$ whenever $f(r)=0$, then $F(x)$ has only real zeros.
\end{lemma}

Combining~\eqref{Cnx-recu} and Lemma~\ref{Liu07}, one can immediately obtain that the zeros of $C_{n-1}(x)$ are real, simple and separate those of $C_{n}(x)$ for $n\geqslant 2$. Comparing~\eqref{Cnx-recu} with Theorem~\ref{lemmaCnx}, we find that
$$\sum_{r=1}^{n-1}\binom{n}{n-r+1}C_{n-r}(x)C_{r-1}(x)=(n-1)xC_{n-1}(x)+x(1-x)\frac{\mathrm{d}}{\mathrm{d}x}C_{n-1}(x),$$
Note that 
\begin{align*}
&nC_{n-1}(x)+\sum_{r=1}^{n-1}\binom{n}{n-r+1}C_{n-r}(x)C_{r-1}(x)\\
&=nC_{n-1}(x)+\sum_{r=1}^{n-1}\binom{n}{r-1}C_{n-r}(x)C_{r-1}(x)\\
&=nC_{n-1}(x)+\sum_{k=0}^{n-2}\binom{n}{k}C_{n-k-1}(x)C_k(x)\\
&=\sum_{k=0}^{n-1}\binom{n}{k}C_{k}(x)C_{n-k-1}(x).
\end{align*}
Then $$\sum_{k=0}^{n-1}\binom{n}{k}C_{k}(x)C_{n-k-1}(x)=(n+(n-1)x)C_{n-1}(x)+x(1-x)\frac{\mathrm{d}}{\mathrm{d}x}C_{n-1}(x).$$
Therefore, by Lemma~\ref{Liu07}, we immediately get the following result.
\begin{proposition}
The binomial convolutions $\sum_{k=0}^{n-1}\binom{n}{k}C_{k}(x)C_{n-k-1}(x)$ have only real zeros.
\end{proposition} 

The $n\times n$ {\it lower Hessenberg matrix} $H_n$ is defined as follows:
\begin{equation*}
H_n=\begin{pmatrix}
h_{11}&h_{12}&0&\cdots&0&0\\
h_{21}&h_{22}&h_{23}&\cdots&0&0\\
h_{31}&h_{32}&h_{33}&\cdots&0&0\\
\vdots&\vdots&\vdots&\ddots&\vdots&\vdots\\
h_{n-1,1}&h_{n-1,2}&h_{n-1,3}&\cdots&h_{n-1,n-1}&h_{n-1,n}\\
h_{n,1}&h_{n,2}&h_{n,3}&\cdots&h_{n,n-1}&h_{n,n}\\
    \end{pmatrix},
  \end{equation*}
where $h_{ij}=0$ if $j>i+1$.
Hessenberg matrices frequently appear in numerical analysis~\cite{Cahill02,Kilic17}.
Setting $H_0=1$,
Cahill et al.~\cite{Cahill02} gave a recursion for
the determinant of the matrix $H_n$ as follows:
\begin{equation}\label{Hn-recu}
\det H_{n}=h_{n,n}\det H_{n-1}+\sum_{r=1}^{n-1}\left((-1)^{n-r}h_{n,r}\prod_{j=r}^{n-1}h_{j,j+1}\det H_{r-1}\right).
\end{equation} 
By induction, comparing~\eqref{Hn-recu} and Theorem~\ref{lemmaCnx}, we discover the following result.
\begin{theorem}\label{thmCn}
For any $n\geqslant 1$, the polynomial $C_n(x)$ can be expressed as the following lower
Hessenberg determinant of order $n$:
 \begin{equation*}
C_{n}(x)=\begin{vmatrix}
x&-1&0&\cdots&0&0\\
C_1(x)&\binom{2}{1}x&-1&\cdots&0&0\\
C_2(x)&\binom{3}{1}C_1(x)&\binom{3}{2}x&\cdots&0&0\\
\vdots&\vdots&\vdots&\ddots&\vdots&\vdots\\
C_{n-2}(x)&\binom{n-1}{1}C_{n-3}(x)&\binom{n-1}{2}C_{n-4}(x)&\cdots&\binom{n-1}{n-2}x&-1\\
C_{n-1}(x)&\binom{n}{1}C_{n-2}(x)&\binom{n}{2}C_{n-3}(x)&\cdots&\binom{n}{n-2}C_1(x)&\binom{n}{n-1}x
    \end{vmatrix}_{n\times n}.
  \end{equation*}
\end{theorem}
When $n=4$, from Theorem~\ref{thmCn}, we see that
 \begin{equation*}
C_4(x)=\begin{vmatrix}
x& -1& 0& 0\\
x& 2x&-1& 0\\
x + 2 x^2& 3 x& 3 x& -1\\
x + 8 x^2 + 6 x^3& 4 x + 8 x^2& 6 x& 4 x
    \end{vmatrix}=x + 22 x^2 + 58 x^3 + 24 x^4.
  \end{equation*}
\section{A class of restricted Stirling permutations and signed permutations}\label{Section3}
\subsection{Context-free grammars}
\hspace*{\parindent}

A {\it context-free grammar} (also known as {\it Chen's grammar}~\cite{Chen93,Dumont96}) $G$ over an alphabet
$V$ is defined as a set of substitution rules replacing a letter in $V$ by a formal function over $V$.
The formal derivative $D_G$ with respect to $G$ satisfies the derivation rules:
$$D_G(u+v)=D_G(u)+D_G(v),~D_G(uv)=D_G(u)v+uD_G(v).$$
So the {\it Leibniz rule} holds:
\begin{equation*}\label{Leibniz}
D_G^n(uv)=\sum_{k=0}^n\binom{n}{k}D_G^k(u)D_G^{n-k}(v).
\end{equation*}

In~\cite{Dumont96}, Dumont obtained the grammar for Eulerian polynomials by using a grammatical labeling of circular permutations.
\begin{proposition}[{\cite[Section~2.1]{Dumont96}}]\label{grammar03}
Let $G=\{a\rightarrow ab, b\rightarrow ab\}$.
Then for $n\geqslant 1$, one has
\begin{equation*}
D_{G}^n(a)=D_{G}^n(b)=b^{n+1}A_n\left(\frac{a}{b}\right).
\end{equation*}
\end{proposition}

In the following three examples, we introduce a new differential operator method, which gives a method to search grammars.
We first establish a connection
between~\eqref{Anx-poly-def} and Proposition~\ref{grammar03}.
\begin{example}\label{ex-Eul}
Setting
\begin{equation*}\label{Tab}
T=x\frac{\mathrm{d}}{\mathrm{d}x},~a=\frac{1}{1-x}~\text{and}~b=\frac{x}{1-x},
\end{equation*}
we obtain $T(a)=T(b)=ab$. By~\eqref{Anx-poly-def}, we have $T^n(a)=D_{G}^n(a)$, where $G=\{a\rightarrow ab, b\rightarrow ab\}$.
\end{example}

The grammar for the second-order Eulerian polynomials was first discovered by Chen-Fu~\cite{Chen17} by using a
grammatical labeling of Stirling permutations. Here we give an original derivation.
\begin{example}\label{ExG4}
Setting $$T=\frac{x}{1-x}\frac{\mathrm{d}}{\mathrm{d}x},~a=\frac{x}{1-x},~b=\frac{1}{1-x},$$
we get $T(a)=ab^2$ and $T(b)=ab^2$. Let $G=\{a\rightarrow ab^2,~b\rightarrow ab^2\}$.
By~\eqref{Cnx-poly-def}, we see that $$D_{G}^n(a)=T^n(a)=b^{2n+1}C_n\left(\frac{a}{b}\right).$$
\end{example}
Following~\cite[p.~29]{Hwang20}, one has
\begin{equation}\label{dxdy}
\left(\frac{\mathrm{d}}{\mathrm{d}y}\right)^n\frac{\mathrm{e}^y}{1-\mathrm{e}^{2y}}=\frac{\mathrm{e}^xB_n(\mathrm{e}^{2x})}{(1-\mathrm{e}^{2x})^{n+1}},
\end{equation}
where $B_n(x)$ is type $B$ Eulerian polynomials.
Note that $\frac{\mathrm{d}}{\mathrm{d}y}=\frac{\mathrm{dx}}{\mathrm{d}y}\frac{\mathrm{d}}{\mathrm{d}x}$ and $x=\mathrm{e}^y$ is the solution of $x=\frac{\mathrm{d}x}{\mathrm{d}y}$.
It follows from~\eqref{dxdy} that
\begin{equation}\label{Bnxsum}
\left(x\frac{\mathrm{d}}{\mathrm{d}x}\right)^n\frac{x}{1-x^2}=\frac{xB_n(x^2)}{(1-x^2)^{n+1}}.
\end{equation}

We now deduce the grammar for the type $B$ Eulerian polynomials $B_n(x)$.
\begin{example}\label{ex-Bn}
Setting $$T=x\frac{\mathrm{d}}{\mathrm{d}x},~a=\frac{x}{\sqrt{1-x^2}}~\text{and}~b=\frac{1}{\sqrt{1-x^2}},$$
we get $T(a)=ab^2,~T(b)=a^2b$.
Let $G=\{a\rightarrow ab^2,~b\rightarrow a^2b\}$.
It follows from~\eqref{Bnxsum} that
\begin{equation*}\label{DG3ab}
D_{G}^n(ab)=T^n(ab)=ab^{2n+1}B_n\left(\frac{a^2}{b^2}\right).
\end{equation*}
\end{example}

In equivalent forms, Dumont~\cite{Dumont80}, Haglund-Visontai~\cite{Haglund12}, Chen-Hao-Yang~\cite{Chen2102} and Ma-Ma-Yeh~\cite{Ma19} all showed that
\begin{equation}\label{DGxyz}
D_G^n(x)=D_G^n(y)=D_G^n(z)=C_n(x,y,z),
\end{equation}
where $G=\{x \rightarrow xyz,~y\rightarrow xyz,~z\rightarrow xyz\}$.
By the change of grammar
$u=x+y+z,~v=xy+yz+zx$ and $w=xyz$,
it is easy to verify that
$D_{G}(u)=3w,~D_{{G}}(v)=2uw,~D_{G}(w)=vw$. So we get a grammar
$H=\{w\rightarrow vw, u\rightarrow 3w, v\rightarrow 2uw\}$.
Recently, Chen-Fu~\cite{Chen22} discovered that for any $n\geqslant 1$, one has
\begin{equation*}\label{Chen22}
C_n(x,y,z)=D_G^n(x)=D_H^{n-1}(w)=\sum_{i+2j+3k=2n+1}\gamma_{n,i,j,k}u^iv^jw^k,
\end{equation*}
where the coefficient $\gamma_{n,i,j,k}$ equals the number of 0-1-2-3 increasing plane trees
on $[n]$ with $k$ leaves, $j$ degree one vertices and $i$ degree two vertices.
Substituting $u\rightarrow x+y+z,~v\rightarrow xy+yz+zx$ and $w\rightarrow xyz$, one has
\begin{equation*}\label{Cnxyz}
C_{n}(x,y,z)=\sum_{i+2j+3k=2n+1}\gamma_{n,i,j,k}(x+y+z)^{i}(xy+yz+zx)^{j}(xyz)^k,
\end{equation*}
\subsection{A class of restricted Stirling permutations}
\hspace*{\parindent}

The following fundamental result will be used in our discussion.
\begin{lemma}
Let $G=\{x \rightarrow xyz,~y\rightarrow xyz,~z\rightarrow xyz\}$. Then we have
\begin{equation}\label{DGyz}
 D_G^n(yz)=\sum_{\sigma\in{\mq}_{n+1}^{(1)}}x^{\plat(\sigma)}y^{\des(\sigma)}z^{\asc(\sigma)},
\end{equation}
where ${\mq}_{n+1}^{(1)}$ is the set of Stirling permutations over the multiset $\{1,2,2,3,3,\ldots,n,n,n+1,n+1\}$.
\end{lemma}
\begin{proof}
We first introduce a grammatical labeling of $\sigma\in{\mq}_{n+1}^{(1)}$ as follows:
\begin{itemize}
\item [\rm ($L_1$)]If $\sigma_i$ is a plateau, then put a superscript label $x$ right after $\sigma_i$.
 \item [\rm ($L_2$)]If $\sigma_i$ is a descent, then put a superscript label $y$ right after $\sigma_i$;
 \item [\rm ($L_3$)]If $\sigma_i$ is an ascent, then put a superscript label $z$ right after $\sigma_i$;
\end{itemize}
The weight of $\sigma$ is given by $w(\sigma)=x^{\plat(\sigma)}y^{\des(\sigma)}z^{\asc(\sigma)}$.
When $n=0$, we have ${\mq}_{1}^{(1)}=\{^z1^y\}$, which corresponds to $yz$. When $n=1$, we have 
${\mq}_{2}^{(1)}=\{^z1^z2^x2^y,^z2^x2^y1^y\}$, which corresponds to $D_G(yz)=x(y^2z+yz^2)$.
When $n=2$, the weighted elements of ${\mq}_{3}^{(1)}$ can be listed as follows:
$$^z1^z2^x2^z3^x3^y,~^z1^z2^z3^x3^y2^y,~^z1^z3^x3^y2^x2^y,~^z3^x3^y1^z2^x2^y,$$
$$^z2^x2^y1^z3^x3^y,~^z2^x2^z3^x3^y1^y,~^z2^z3^x3^y2^y1^y,~^z3^x3^y2^x2^y1^y.$$
The sum of weights of those elements is given as follows: 
$$D_G^2(yz)=x^2y^3z+4x^2y^2z^2+xy^3z^2+x^2yz^3+xy^2z^3.$$ So the result holds for $n=0,1,2$.
Suppose we get all labeled permutations in $\mq_{n-1}^{(1)}$, where $n\geqslant 3$. Let
$\sigma'$ be obtained from $\sigma\in \mq_{n-1}^{(1)}$ by inserting the pair $nn$.
Then the changes of labeling are illustrated as follows:
$$\cdots\sigma_i^x\sigma_{i+1}\cdots\mapsto \cdots\sigma_i^zn^xn^y\sigma_{i+1}\cdots;$$
$$\cdots\sigma_i^y\sigma_{i+1}\cdots\mapsto \cdots\sigma_i^zn^xn^y\sigma_{i+1}\cdots;$$
$$\cdots\sigma_i^z\sigma_{i+1}\cdots\mapsto \cdots\sigma_i^zn^xn^y\sigma_{i+1}\cdots.$$

In each case, the insertion of the string $nn$ corresponds to the operator $D_G$.
By induction, the action of the formal derivative $D_G$
on the set of weighted Stirling permutations in $\mq_{n-1}^{(1)}$ gives the set of 
weighted Stirling permutations in $\mq_{n}^{(1)}$, and so~\eqref{DGyz} holds. 
\end{proof}

Define $$E_n(x,y,z)=\sum_{\sigma\in{\mq}_{n}^{(1)}}x^{\plat(\sigma)}y^{\des(\sigma)}z^{\asc(\sigma)}.$$
By~\eqref{DGyz}, we see that
\begin{equation}\label{Dumont801}
E_{n+1}(x,y,z)=xyz\left(\frac{\partial}{\partial x}+\frac{\partial}{\partial y}+\frac{\partial}{\partial z}\right)E_n(x,y,z),~E_1(x,y,z)=yz.
\end{equation}
Let $E_n(x,y,z)=\sum_{i\geqslant 1,j\geqslant 1}E_{n,i,j}x^iy^jz^{2n-i-j}$ for $n\geqslant 2$. It follows from~\eqref{Dumont801} 
that 
\begin{equation}\label{Dumont802}
E_{n+1,i,j}=iE_{n,i,j-1}+jE_{n,i-1,j}+(2n-i-j+2)E_{n,i-1,j-1},
\end{equation}
with $E_{1,0,1}=1$ and $E_{1,i,j}=0$ for all $(i,j)\neq(0,1)$.
By~\eqref{DGxyz}, we find that 
$$C_{n+1}(x,y,z)=D_G^{n}(xyz)=\sum_{k=0}^n\binom{n}{k}D_G^k(x)D_G^{n-k}(yz)=\sum_{k=0}^n\binom{n}{k}C_k(x,y,z)D_G^{n-k}(yz).$$
So the following result is immediate.
\begin{theorem}
We have 
\begin{equation*}\label{Cxyz}
C_{n+1}(x,y,z)=\sum_{k=0}^n\binom{n}{k}C_k(x,y,z)E_{n-k+1}(x,y,z).
\end{equation*}
\end{theorem}

It is routine to check that $\#{\mq}_{n+1}^{(1)}=E_{n+1}(1,1,1)=(2n)!!$.
Let $\pm[n]=[n]\cup\{\overline{1},\overline{2},\ldots,\overline{n}\}$, where $\overline{i}=-i$.
The {\it hyperoctahedral group} $\msn^B$ is the group of signed permutations on $\pm[n]$ with the property that $\pi\left(\overline{i}\right)=-\pi(i)$ for all $i\in [n]$.
It is well known that $\#\msn^B=(2n)!!$. Therefore, 
$$\#{\mq}_{n+1}^{(1)}=\#\msn^B.$$
A natural idea is therefore to investigate the connections between ${\mq}_{n+1}^{(1)}$ and $\msn^B$.
\subsection{Six-variables polynomials over restricted Stirling permutations}
\hspace*{\parindent}

For $\sigma\in\mqn$, except where explicitly stated, we always assume that $\sigma_0=\sigma_{2n+1}=0$.
Let
\begin{align*}
\operatorname{Ap}(\sigma)&=\{\sigma_i\mid \sigma_{i-1}<\sigma_{i}=\sigma_{i+1}~\&~ 2\leqslant i\leqslant 2n-1\},\\
\operatorname{Lap}(\sigma)&=\{\sigma_i\mid \sigma_{i-1}<\sigma_{i}=\sigma_{i+1}~\&~ 1\leqslant i\leqslant 2n-1\}
\end{align*}
be the sets of ascent-plateaux and left ascent-plateaux of $\sigma$, respectively. Let $\ap(\sigma)$ and $\lap(\sigma)$ be the numbers of ascent-plateaux and left ascent-plateaux, respectively.
The {\it ascent-plateau polynomials} $M_n(x)$
and the {\it left ascent-plateau polynomials} $W_n(x)$ can be defined as follows:
\begin{equation*}
\begin{split}
M_n(x)&=\sum_{\sigma\in\mqn}x^{\ap(\sigma)},~
W_n(x)=\sum_{\sigma\in\mqn}x^{\lap(\sigma)}.
\end{split}
\end{equation*}
From~\cite[Proposition~1]{MaYeh17}, we see that
\begin{equation*}\label{Convo01}
\begin{split}
2^nA_n(x)&=\sum_{i=0}^n\binom{n}{i}W_i(x)W_{n-i}(x),~
B_n(x)=\sum_{i=0}^n\binom{n}{i}M_i(x)W_{n-i}(x),
\end{split}
\end{equation*}
where $B_n(x)$ is the type $B$ Eulerian polynomial over the hyperoctahedral group $\msn^B$.

Let $\pi=\pi(1)\pi(2)\cdots\pi(n)\in \msn^B$.
It should be noted that the $n$ letters
appearing in the cycle notation of $\pi\in \msn^B$ are the
letters $\pi(1),\pi(2),\ldots,\pi(n)$.
We say that $i$ is an {\it excedance} (resp.~{\it anti-excedance}, {\it fixed point}, {\it singleton}) of $\pi$ if $\pi(|\pi(i)|)>\pi(i)$ (resp.~$\pi(|\pi(i)|)<\pi(i)$, $\pi(i)=i$, $\pi(i)=\overline{i}$).
Let $\exc(\pi)$ (resp.~$\aexc(\pi)$, $\fix(\pi)$, $\single(\pi)$, $\negg(\pi)$, $\cyc(\pi)$) be the number of excedances (resp.~anti-excedances, fixed points, singletons, negative entries, cycles) of $\pi$. 
\begin{example}
The signed permutation $\pi=68\overline{3}15\overline{7}24\overline{9}$ can be written as
$(\overline{9})(\overline{3})(1,6,\overline{7},2,8,4)(5)$. Moreover, $\pi$ has two
singleton $\overline{9}$ and $\overline{3}$, one fixed point $5$, and $\pi$ has $3$ excedances, $3$ anti-excedances and $3$ negative entries.
\end{example}
According to~\cite[Corollary~3.16]{Brenti94}, we have 
$$B_n(x)=\sum_{\pi\in\msn^B}x^{\exc(\sigma)+\single(\sigma)}=\sum_{\pi\in\msn^B}x^{\exc(\sigma)+\fix(\sigma)}.$$

\begin{lemma}[{\cite[Lemma~5.3]{Ma24}}]
Let $p$ and $q$ be two given parameters.
If
\begin{equation}\label{Ixyst-G}
G=\{I\rightarrow qI(t+sp),s\rightarrow (1+p)xy,t\rightarrow (1+p)xy,x\rightarrow (1+p)xy,y\rightarrow (1+p)xy\},
\end{equation}
then we have $$D_{G}^n(I)=I\sum_{\sigma\in \msn^B}x^{\exc(\sigma)}y^{\aexc(\sigma)}s^{\single(\sigma)}t^{\fix(\sigma)}p^{\negg(\sigma)}q^{\cyc(\sigma)}.$$
\end{lemma}

We need more definitions.
Let $w=w_1w_2\cdots w_n$ be a word of length $n$, where $w_i$ are all integers. We say that $w_i$ is 
\begin{itemize}
  \item [$\bullet$] an {\it even indexed entry} if the 
first appearance of $w_i$ (when we read $w$ from left to right) occurs at an even position;
 \item [$\bullet$] A {\it right-to-left minimum} of $w$ is an element $w_i$ such that $w_i\leqslant w_j$
for every $j\in\{i+1,i+2,\ldots, n\}$ and $w_i$ is the first appearance from left to right;
  \item [$\bullet$]  A {\it left-to-right minimum} of $w$ is an element $w_i$ such that $w_i< w_j$
for every $j\in\{1,2,\ldots,i-1\}$ or $i=1$;
\item [$\bullet$] A {\it block} of $w$ is defined as a maximal substring that begins with a left-to-right minimum and contains no other left-to-right minima.
 \end{itemize}
Let $\operatorname{Even}(\sigma)$ and $\operatorname{Rlmin}(\sigma)$ be the sets of even indexed entries and right-to-left minima of $\sigma$, respectively. For instance, 
$$\operatorname{Even}(4\textbf{5}541{\textbf{2}}2{\textbf{3}}773{\textbf{6}}6)=\{2,3,5,6\},
~\operatorname{Rlmin}(4554\textbf{1}\textbf{2}2\textbf{3}773\textbf{6}6)=\{1,2,3,6\}.$$

It is easily derived by induction that any Stirling permutation $\sigma$ in $\mq_n^{(1)}$ has a unique decomposition as a sequence of blocks. 
More precisely, if $i\geqslant 2$ is a left-to-right minimum of $\sigma$, then the block $[i,i]_{\sigma}$ 
is the consecutive string of elements of $\sigma$ between and including the two occurrences of $i$ in $\sigma$. 
Clearly, there exists one block begins with the elements $1$.
For example, the block decomposition of $883466439912552771$ is given by $[88][34664399][1255277]$.
We use $\even(\sigma)$, $\operatorname{rlmin}(\sigma)$, $\operatorname{lrmin}(\sigma)$ and $\operatorname{bk}_2(\sigma)$ to denote the 
numbers of even indexed entries, right-to-left minima, left-to-right minima and blocks of size exactly 2 of $\sigma$, respectively.
When $\sigma_i\sigma_{i+1}$ constitutes a block with size exactly 2, it is clear that $\sigma_i=\sigma_{i+1}$. For example, $\operatorname{bk}_2(33221)=2$.
The block statistic has been studied by Kuba-Panholzer~\cite{Kuba21} and Remmel-Wilson~\cite{Remmel14}.

We need some new definitions. 
\begin{definition}
For $\sigma\in\mq_{n}^{(1)}$, we say that an entry $\sigma_i$ is 
\begin{itemize}
  \item [$\bullet$] a {\it proper ascent-plateau} if $\sigma_i$ is an ascent plateau, but it is not a right-to-left minimum;
  \item [$\bullet$] an {\it improper ascent-plateau} if $\sigma_i\in\operatorname{Ap}(\sigma)\cap \operatorname{Rlmin}(\sigma)$;
   \item [$\bullet$] a {\it trace} if there exists $2\leqslant k\leqslant n$, when the subword of $\sigma$ restricted to $\{1,2,2,\ldots,k,k\}$, 
   $\sigma_i$ is an {\it improper ascent-plateau} of this subword or $\sigma_i$ is the second entry that appears in a block of size exactly 2. 
 \end{itemize}
\end{definition}
Let $\operatorname{Pap}(\sigma),\operatorname{Impap}(\sigma)$ and $\operatorname{Trace}(\sigma)$ be the sets of proper ascent-plateaux, improper ascent-plateaux and traces of $\sigma$, respectively.
For instance, $$\operatorname{Pap}(8845541{2}2{3}773{6}6)=\{5,7\},$$
$$\operatorname{Impap}(8845541{2}2{3}773{6}6)=\{2,6\},$$
$$\operatorname{Trace}(8845541{2}2{3}773{6}6)=\{2,3,4,6,8\}.$$
It is clear that $$\operatorname{AP}(\sigma)=\operatorname{Pap}(\sigma)\cup \operatorname{Impap}(\sigma),~\operatorname{Pap}(\sigma)\cap \operatorname{Impap}(\sigma)=\emptyset.$$
We use $\operatorname{pap}(\sigma)$ and $\operatorname{impap}(\sigma)$ to denote the 
numbers of proper and improper ascent-plateaux of $\sigma$, respectively. Hence $\ap(\sigma)=\operatorname{pap}(\sigma)+\operatorname{impap}(\sigma)$.
Moreover, set $$\operatorname{cap}(\sigma)=n-\operatorname{pap}(\sigma)-\operatorname{impap}(\sigma)-\operatorname{bk}_2(\sigma)=n-\ap(\sigma)-\operatorname{bk}_2(\sigma).$$

We can now conclude the main result of this section.
\begin{theorem}\label{BNQN}
We have
$$\sum_{\sigma\in \mq_{n+1}^{(1)}}x^{\operatorname{pap}(\sigma)}y^{\operatorname{cap}(\sigma)}s^{\operatorname{impap}(\sigma)}t^{\operatorname{bk}_2(\sigma)}p^{\even(\pi)}q^{\operatorname{tr}(\sigma)}
=\sum_{\pi\in \msn^B}x^{\exc(\pi)}y^{\aexc(\pi)}s^{\single(\pi)}t^{\fix(\pi)}p^{\negg(\pi)}q^{\cyc(\pi)}.$$
\end{theorem}
\begin{proof}
Recall that ${\mq}_{n}^{(1)}$ is the set of Stirling permutations over the multiset $\{1,2,2,3,3,\ldots,n,n\}$.
We first introduce a grammatical labeling of $\sigma\in \mq_n^{(1)}$ as follows:
\begin{itemize}
  \item [\rm ($L_1$)]we use the superscript $I$ to mark the first position (just before $\sigma_1$) 
and the last position (at the end of $\sigma$), and denoted by $\overbrace{\sigma_1\sigma_2\cdots \sigma_{2n-1}}^{I}$;
 \item [\rm ($L_2$)]if $\sigma_i$ is a proper ascent-plateau, then we label the two positions just before and right after $\sigma_i$ by a subscript label $x$;
\item [\rm ($L_3$)]if $\sigma_i$ is an improper ascent-plateau, then we label the two positions just before and right after $\sigma_i$ by a subscript label $s$;
\item [\rm ($L_4$)]if $\sigma_i\sigma_{i+1}$ constitute a block with size exactly 2, then we label the two positions just before and right after $\sigma_{i+1}$ by a subscript label $t$;
\item [\rm ($L_5$)] except the above labeled positions, there still has an even number of positions, and we use a $y$ to label pairwise nearest elements from left to right;
\item [\rm ($L_6$)]we attach a superscript label $p$ to every even indexed entry;
\item [\rm ($L_7$)]we attach a subscript label $q$ to each trace.
\end{itemize}
Note that the weight of $\sigma$ is given by $$w(\sigma)=x^{\operatorname{pap}(\sigma)}y^{\operatorname{cap}(\sigma)}s^{\operatorname{impap}(\sigma)}t^{\operatorname{bk}_2(\sigma)}p^{\even(\pi)}q^{\operatorname{tr}(\sigma)}.$$
The element in $\mq_1^{(1)}$ can be labeled as $\overbrace{1}^I$.
The labeled elements in $\mq_2^{(1)}$ can be listed as follows:
$$\overbrace{1\underbrace{2^p_q}_s2}^I,~\overbrace{2\underbrace{2_q}_t1}^I.$$
Let $G$ be given by~\eqref{Ixyst-G}.
Note that $D_{G}(I)=qI(t+sp)$.
Hence the result holds for $n=1$. We proceed by induction.
Suppose we get all labeled permutations in $\sigma\in \mq_{n-1}^{(1)}$, where $n\geqslant 2$. Let
$\widehat{{\sigma}}$ be obtained from $\sigma$ by inserting the string $nn$.
There are six ways to label the inserted string and to relabel some elements of $\sigma$:
\begin{itemize}
  \item [\rm ($c_1$)] by the definition of trace, we never need to relabel the subscript label $q$.
  \item [\rm ($c_2$)]if $nn$ is inserted immediately before or after $\sigma$, then the changes of labeling are illustrated as follows:
$$\overbrace{\sigma}^I \rightarrow\overbrace{n\underbrace{n_q}_t\sigma}^I,\quad \overbrace{\sigma}^I\rightarrow\overbrace{\sigma \underbrace{n^p_q}_sn}^I.$$
 \item [\rm ($c_3$)]if $nn$ is inserted immediately before or after an element with label $t$. Since $ii$ forms a block of size 2 which means the entries before $ii$ are all bigger than $i$, thus the first $i$ located at an odd position. The changes of labeling are illustrated as follows:
$$\cdots{i\underbrace{i_q}_t}\cdots \rightarrow \cdots i{\underbrace{n^p}_xn{\underbrace{i_q}_y}}\cdots,\quad 
\cdots{i\underbrace{i_q}_t}\cdots \rightarrow \cdots \cdots i\overbrace{{i_q}\underbrace{n}_xn}^y\cdots\cdots.$$
\item [\rm ($c_4$)]if $nn$ is inserted immediately before or after an element with label $s$, then the changes of labeling are illustrated as follows:
$$\cdots\underbrace{i^p_q}_si\cdots \rightarrow \cdots\underbrace{ n^p}_xn \overbrace{i^p_q}^y i\cdots,\quad 
\cdots\underbrace{i^p_q}_si\cdots \rightarrow \cdots \overbrace{i^p_q \underbrace{n}_xn}^y  i\cdots.$$
\item [\rm ($c_5$)]if $nn$ is inserted immediately before or after an element with label $x$, then the changes of labeling are illustrated as follows:
$$\cdots\underbrace{i^p}_xi\cdots \rightarrow \cdots\underbrace{n^p}_xn\overbrace{i^p}^yi\cdots,\quad 
\cdots\underbrace{i^p}_xi\cdots \rightarrow \cdots \overbrace{i^p \underbrace{n}_xn}^yi\cdots;$$
$$\cdots\underbrace{i}_xi\cdots \rightarrow \cdots\underbrace{n}_xn\overbrace{i^p}^yi\cdots,\quad 
\cdots\underbrace{i}_xi\cdots \rightarrow \cdots \overbrace{i \underbrace{n^p}_xn}^yi\cdots.$$
\item [\rm ($c_6$)]if the string $nn$ is inserted into either position of a pair labeled by $y$, then the first $n$ always get a label $x$ and there still two positions get a label $y$. For each pair of positions labeled by $y$, there is an odd number of entries between them. So we get the substitution rule $y \rightarrow (1+p)xy$.
\end{itemize}

In each case, the insertion of $nn$ corresponds to one substitution rule in $G$. Therefore, the action of $D_{G}$
on the set of weights of Stirling permutations in $\mq_{n-1}^{(1)}$  gives the set of weights of Stirling permutations in $\mq_{n}^{(1)}$.
This yields the desired result, and so we complete the proof.
\end{proof}

We now collect several well-studied Eulerian-type polynomials.  Let
  $$B_n(x,p,q)=\sum_{\pi\in \msn^B}x^{\exc(\pi)+\single(\pi)}p^{\negg(\pi)}q^{\cyc(\pi)}$$
 be a $(p,q)$-Eulerian polynomials of type $B$. 
The {\it types $A$ and $B$ derangement polynomials} are respectively defined by
$$d_n(x)=\sum_{\pi\in \mathcal{D}_n}x^{\exc(\pi)},~d_n^B(x)=\sum_{\pi\in \mathcal{D}_n^B}x^{\exc(\pi)},$$
where $\mathcal{D}_n=\{\pi\in \msn: \fix(\pi)=0\}$ and $\mathcal{D}_n^B=\{\pi\in \msn^B: \fix(\pi)=0\}$.
These polynomials have been extensively studied, see~\cite{Chow09,Zeng2020,Ma24} for recent progress on this subject.

\begin{corollary}
We have 
$$B_n(x,p,q)=\sum_{\sigma\in \mq_{n+1}^{(1)}}x^{\operatorname{ap}(\sigma)}p^{\even(\pi)}q^{\operatorname{tr}(\sigma)},$$
$$d_n(x)=\sum_{\substack{\pi\in \mq_{n+1}^{(1)}\\ \operatorname{bk}_2(\pi)=0\\\even(\pi)=0}}x^{\operatorname{pap}(\pi)},~d_n^B(x)=\sum_{\substack{\pi\in \mq_{n+1}^{(1)}\\ \operatorname{bk}_2(\pi)=0}}x^{\operatorname{pap}(\pi)}.$$
\end{corollary}

The {\it $(p,q)$-Eulerian polynomials} $A_n(x,p,q)$ are defined by
\begin{equation*}\label{Anxpq-def}
A_n(x,p,q)=\sum_{\pi\in\msn}x^{\exc(\pi)}p^{\fix(\pi)}q^{\cyc(\pi)}.
\end{equation*}
Combining~\cite[Theorem~5.2]{Ma24} and Theorem~\ref{BNQN}, we get the following result.
\begin{corollary}\label{corl2}
We have
\begin{equation*}\label{Bxy-EGF}
\sum_{\sigma\in \mq_{n+1}^{(1)}}x^{\operatorname{pap}(\sigma)}y^{\operatorname{cap}(\sigma)}s^{\operatorname{impap}(\sigma)}t^{\operatorname{bk}_2(\sigma)}p^{\even(\pi)}q^{\operatorname{tr}(\sigma)}
=(1+p)^ny^nA_n\left(\frac{x}{y},\frac{t+sp}{y+py},q\right).
\end{equation*}
\end{corollary}
\section{Box sorting algorithm and standard Young tableaux}\label{Section4}
\subsection{Preliminaries}
\hspace*{\parindent}

The {\it Weyl algebra} $W$ is the unital algebra generated by two symbols $D$ and $U$ satisfying the commutation relation
$DU-UD=I$,
where $I$ is the identity which we identify with ``1".
Any word $w$ in the letters $U,D$ can always be brought
into {\it normal ordered form} where all letters $D$ stand to the right of all the letters $U$.
A famous example of $W$ is given by the substitution:
$D\rightarrow \frac{\mathrm{d}}{\mathrm{d}x},~U\rightarrow x$. Except where otherwise indicated, we always let $D=\frac{\mathrm{d}}{\mathrm{d}x}$.
The expansion of $(xD)^n$ has been studied as early as 1823 by Scherk~\cite[Appendix~A]{Blasiak10}. He found that
\begin{equation}\label{Stirling-def}
(xD)^n=\sum_{k=0}^n\Stirling{n}{k}x^kD^k.
\end{equation}
According to~\cite[Proposition~A.2]{Blasiak10}, one has
\begin{equation*}\label{stirling-def}
(\mathrm{e}^xD)^n=\mathrm{e}^{nx}\sum_{k=0}^n\stirling{n}{k}D^k,
\end{equation*}
where $\stirling{n}{k}$ is the (signless) Stirling number of the first
kind, i.e.,
the number of permutations of $[n]$ with $k$ cycles.
Many generalizations of~\eqref{Stirling-def} occur naturally in quantum physics, combinatorics and algebra.
The reader is referred to~\cite{Blasiak10,Schork21} for surveys on this topic.

Throughout this paper, we always let $c:=c(x)$ and $f:=f(x)$ be two smooth functions in the indeterminate $x$.
We adopt the convention that $c_k=D^kc$ and $\mathbf{f}_k=D^kf$ for $k\geqslant 0$. Set $\mathbf{f}_0=f$ and $c_0=c$, where $D=\frac{\mathrm{d}}{\mathrm{d}x}$.
The first few $(cD)^nf$ are given as follows:
\begin{align*}
(cD)f&=(c)  {\mathbf{f}}_1,~(cD)^2f=(c c_1 )  {\mathbf{f}}_1 +(c^2 )  {\mathbf{f}}_2,\\
(cD)^3f&=(c c_1^2  +c^2 c_2 )  {\mathbf{f}}_1 +(3c^2 c_1 )  {\mathbf{f}}_2 +(c^3 )  {\mathbf{f}}_3,\\
(cD)^4f&=(c c_1^3  +4c^2 c_1 c_2  +c^3 c_3 )  {\mathbf{f}}_1 +(7c^2 c_1^2  +4c^3 c_2 )  {\mathbf{f}}_2 +(6c^3 c_1 )  {\mathbf{f}}_3 +(c^4 )  {\mathbf{f}}_4.
\end{align*}

For $n\geqslant 1$, we define
\begin{equation}\label{Ank-def}
(cD)^nf =\sum_{k=1}^nF_{n,k}\mathbf{f}_k.
\end{equation}
Note that $F_{n,k}=F_{n,k}(c,c_1,\ldots,c_{n-k})$ is a function of $c,c_1,\ldots,c_{n-k}$.
In particular, $F_{1,1}=c$, $F_{2,1}=cc_1$ and $F_{2,2}=c^2$.
Clearly, $F_{n+1,1}=cDF_{n,1}$, $F_{n,n}=c^n$ and for $2\leqslant k\leqslant n$, we have
$$F_{n+1,k}=cF_{n,k-1}+cDF_{n,k}.$$

By induction, Comtet~\cite{Comtet73} found an explicit formula of $F_{n,k}$.
Recently, Briand-Lopes-Rosas~\cite{Briand20}
gave a survey on the combinatorial and arithmetic properties of the coefficients $F_{n,k}$, which can be summarized as follows.
\begin{proposition}[{\cite{Briand20}}]\label{prop02}
Let $F_{n,k}$ be defined by~\eqref{Ank-def}.
There exist positive integers $a(n,\lambda)$ such that
\begin{equation*}
F_{n,k}=\sum_{\lambda\vdash n-k} a(n,\lambda)c^{n-\ell(\lambda)}c_{\lambda},
\end{equation*}
where $\lambda$ runs over all partitions of $n-k$. The Stirling numbers of the first and second kinds, and the Eulerian numbers can be respectively expressed as follows:
\begin{equation*}\label{Fnk}
\stirling{n}{k}=\sum_{\lambda\vdash n-k} a(n,\lambda),~
\Stirling{n}{k}=a(n,1^{n-k}),~
\Eulerian{n}{k}=\sum_{\ell(\lambda)=n-k} a(n,\lambda).
\end{equation*}
\end{proposition}

In~\cite{Han23}, Han-Ma first gave a simple proof of Comtet's formula via inversion sequences,
and then introduced $k$-Young tableaux and their $g$-indices. Using the indispensable $k$-Young tableaux, Han-Ma obtained
a unified combinatorial interpretations of Eulerian polynomials and second-order Eulerian
polynomials. In the rest of this paper, we shall give a substantial and original improvement of this idea.
\subsection{Box sorting algorithm}
\hspace*{\parindent}

Rota~\cite{Rota92} once said ``I will tell you shamelessly what my bottom line is: It is placing
balls into boxes". As discussed before, we always let $D=\frac{\mathrm{d}}{\mathrm{d}x}$ and $c=c(x)$.
In order to study the powers of $cD$, we shall introduce the box sorting algorithm.

An {\it ordered weak set partition} of $[n]$ is a list of pairwise disjoint subsets (maybe empty) of $[n]$
such that the union of these subsets is $[n]$. These subsets are called the {\it parts} of the partition.
A {\it weak composition} $\alpha$ of an integer $n$, denoted by $\alpha\models n$, with $m$ parts is a way of writing
$n$ as the sum of any sequence $\alpha=(\alpha_1,\alpha_2,\ldots,\alpha_m)$ of nonnegative integers.
Given $\alpha\models n$. The {\it Young weak composition diagram} of $\alpha$, also denoted by $\alpha$,
is the left-justified array of $n$ boxes
with $\alpha_i$ boxes in the $i$-th row. We follow the French convention, which
means that we number the rows from bottom to top, and the columns from left to
right. The box in the $i$-th row and $j$-th column is denoted by the pair $(i,j)$.
A {\it Young weak composition tableau} ($\YWCT$, for short) of $\alpha$ is obtained by placing the integers $\{1,2,\ldots,n\}$ into $n$ boxes of the diagram
such that each of those integers is used exactly once.
We will often identify an ordered weak set partition with the corresponding $\YWCT$.
It should be noted that there may be some empty boxes in $\YWCT$. In the following discussion,
we always put a special column of $n+1$ boxes at the left of $\YWCT$ or $\SYT$, and labelled by $0,1,2,\ldots, n$ from bottom to top, see Figure~\ref{fig3} and Table~\ref{tab:dummy-1} for instances.
\begin{figure}[!ht]
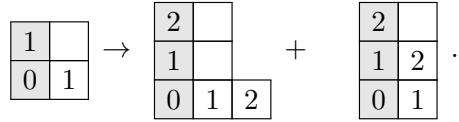
\label{fig3}
\renewcommand{\arraystretch}{2}
\begin{center}
\begin{tabular}{c}
    \begin{ytableau}
    *(gray!20) 1& \\
    *(gray!20) 0 & 1 \\
    \end{ytableau}
\end{tabular}$\rightarrow$
\begin{tabular}{c}
    \begin{ytableau}
    *(gray!20) 2& \\
    *(gray!20) 1& \\
    *(gray!20) 0 & 1&2 \\
    \end{ytableau}\\
\end{tabular}+~~
~~\quad \begin{tabular}{c}
    \begin{ytableau}
    *(gray!20) 2& \\
    *(gray!20) 1& 2\\
    *(gray!20) 0 & 1 \\
    \end{ytableau}
\end{tabular}.
\end{center}
\caption{An illustration of the change of weights: $cc_1\rightarrow c^2c_2+cc_1^2$}\label{fig3}
\end{figure}

The following label schema is fundamental.
\begin{LSp}
Let $p$ be an ordered weak set partition of $[n]$.
We give a labeling of $p$ as follows.
Label the $i$-th subset by the subscript $c_{(i-1)}$,
label a subset with $i$ elements by a superscript $c_i$, where $i\geqslant 1$. Moreover, if the $i$-th subset is empty, we always label it by a superscript $c$.
The {\it weight} of $p$ is defined as the product of the superscript labels.
\end{LSp}

We can rewrite $(cD)^n c$ as follows:
\begin{equation}\label{cdn}
\left(c_{(n)}D_{(n)}\right)\left(c_{(n-1)}D_{(n-1)}\right) \cdots \left(c_{(2)}D_{(2)}\right) \left(c_{(1)}D_{(1)}\right) c_{(0)},
\end{equation}
where $c_{(0)}=c_{(i)}=c$ and $D_{(i)}=D$ for all $i\in [n]$.
A crucial observation is that the differential operator $D_{(i)}$ in~\eqref{cdn} can only applied to $c_{(k)}$, where $0\leqslant k\leqslant i-1$.
When $n=1$, we have $(cD)c=\left(c_{(1)}D_{(1)}\right) c_{(0)}=cc_1$. When $n=2$,
we have
$(cD)^2c=\left(c_{(2)}D_{(2)}\right) \left(c_{(1)}D_{(1)}\right) c_{(0)}=cc_1^2+c^2c_2$.

Next, we introduce the box sorting algorithm, designed to transform a term in the expansion of~\eqref{cdn} into an ordered weak set partition, which can be represented by a $\YWCT$.
When a new term $\left(c_{(i)}D_{(i)}\right)$ is multiplied,
the procedure can be summarized as follows:
\begin{itemize}
  \item When applying $D_{(i)}$ to $c_{(j)}$, it corresponds to the insertion of the element $i$ into the box with the subscript $c_{(j)}$;
  \item Multiplying by $c_{(i)}$ corresponds to the opening of a new empty box $\{\}^{c}_{c_{(i)}}$.
\end{itemize}

We now provide a detailed description of the {\it box sorting algorithm}.
Start with an empty box $(\{\}^c_{c_{(0)}})$. We proceed as follows:
\begin{itemize}
  \item [$BS1$:] When $n=1$, we first insert the element $1$ to the empty box, which corresponds to the operation $D_{(1)}(c_{(0)})$.
We then open a new empty box, which corresponds to the multiplication by $c_{(1)}$. Thus we get $(\{1\}_{c_{(0)}}^{c_1},\{\}^c_{c_{(1)}})$.
  \item [$BS2$:] When $n=2$, we distinguish two cases: $(i)$ we first insert the element $2$ into the first box $\{1\}^{c_1}_{c_{(0)}}$, which corresponds to apply the operation $D_{(2)}$ to  $c_{(0)}$. We then open a new empty box, which corresponds to the multiplication by $c_{(2)}$; $(ii)$
We first insert the element $2$ into the empty box $\{\}^c_{c_{(1)}}$, which corresponds to apply the operation $D_{(2)}$ to $c_{(1)}$. We then open a new empty box, which corresponds to the multiplication by $c_{(2)}$.
Therefore, we get the following correspondences between ordered weak set partitions and their weights:
$$c^2c_2\leftrightarrow (\{1,2\}^{c_2}_{c_{(0)}},\{\}^c_{c_{(1)}},\{\}^c_{c_{(2)}}),~~cc_1^2\leftrightarrow(\{1\}^{c_1}_{c_{(0)}},\{2\}^{c_1}_{c_{(1)}},\{\}^c_{c_{(2)}}).$$
The process from $BS1$ to $BS2$ can be illustrated by Figure~\ref{fig3}.
  \item [$BS3$:] If all of the elements $[i-1]$ have already been inserted, then we consider the insertion of $i$, where $i\geqslant 3$.
Suppose that we insert the element $i$ into the $k$-th box, which has the label $\{\}^{c_\ell}_{c_{(k-1)}}$, where $1\leqslant k\leqslant i$.
Then this insertion corresponds to apply $D_{(i)}$ to $c_{(k-1)}$, and the labels of the $k$-th box become $\{\}^{c_{\ell}+1}_{c_{(k-1)}}$.  We then open a new empty box, which corresponds to the multiplication by $c_{(i)}$. When $i=3$, see Figures~\ref{fig4} and~\ref{fig5} for illustrations, where each empty box in the first column of a $\YWCT$ corresponds to an empty subset.
\end{itemize}

\begin{figure}[!ht]
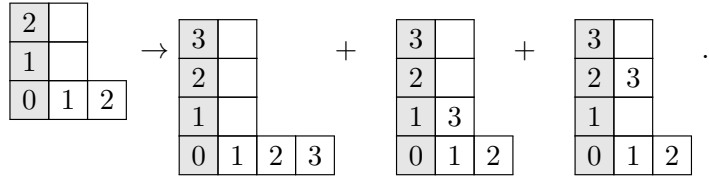
\label{fig4}
\renewcommand{\arraystretch}{2}
\begin{center}
\begin{tabular}{c}
    \begin{ytableau}
    *(gray!20) 2& \\
    *(gray!20) 1& \\
    *(gray!20) 0 & 1&2 \\
    \end{ytableau}\\
\end{tabular}$\rightarrow$
  \begin{ytableau}
    *(gray!20) 3&\\
    *(gray!20) 2&\\
    *(gray!20) 1&\\
    *(gray!20) 0 & 1 & 2 &3  \\
    \end{ytableau}+~~\quad
   \begin{ytableau}
    *(gray!20) 3&\\
    *(gray!20) 2&\\
    *(gray!20) 1 & 3\\
    *(gray!20) 0 & 1 & 2  \\
    \end{ytableau}+~~\quad
    \begin{ytableau}
    *(gray!20) 3&\\
    *(gray!20) 2 & 3\\
    *(gray!20) 1 &\\
    *(gray!20) 0 & 1 & 2  \\
    \end{ytableau}~~.
\end{center}
\caption{The insertion of $3$ into $(\{1,2\}^{c_2}_{c_{(0)}},\{\}^c_{c_{(1)}},\{\}^c_{c_{(2)}})$}\label{fig4}
\end{figure}
\begin{figure}[!ht]
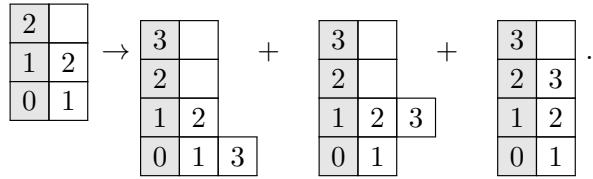
\label{fig5}
\renewcommand{\arraystretch}{2}
\begin{center}
\begin{tabular}{c}
    \begin{ytableau}
    *(gray!20) 2& \\
    *(gray!20) 1& 2\\
    *(gray!20) 0 & 1 \\
    \end{ytableau}\\
\end{tabular}$\rightarrow$
   \begin{ytableau}
    *(gray!20) 3&\\
    *(gray!20) 2&\\
    *(gray!20) 1 & 2\\
    *(gray!20) 0 & 1 & 3  \\
    \end{ytableau}+~~\quad
    \begin{ytableau}
    *(gray!20) 3&\\
    *(gray!20) 2 & \\
    *(gray!20) 1 &2&3\\
    *(gray!20) 0 & 1   \\
    \end{ytableau}+~~\quad
  \begin{ytableau}
    *(gray!20) 3&\\
    *(gray!20) 2&3\\
    *(gray!20) 1&2\\
    *(gray!20) 0 & 1   \\
    \end{ytableau}~~~.
\end{center}
\caption{The insertion of $3$ into $(\{1\}^{c_1}_{c_{(0)}},\{2\}^{c_1}_{c_{(1)}},\{\}^c_{c_{(2)}})$}\label{fig5}
\end{figure}

\begin{definition}\label{defOWP}
Let $\operatorname{OWP}_{n}$
denote the collection of ordered weak set partitions of $[n]$ into $n+1$ blocks $B_0\cup B_1\cup\cdots \cup B_n$ for which the following conditions hold:
$(a)$ $1\in B_0$; $(b)$
if $B_i$ is nonempty, then its minimum larger than $i$, where $1\leqslant i\leqslant n$.
\end{definition}

Given $p\in\operatorname{OWP}_{n}$. It is clear that the $(n+1)$-th block $B_n$ must be empty.
Denote by $w_i(p)$ the number of blocks in $p$ with $i$ elements.
The {\it weight function} of $p$ is defined by
\begin{equation}\label{wp}
w(p)=\prod_{i=0}^nc_i^{w_i(p)}.
\end{equation}
By the box sorting algorithm, we immediately get the following result.
\begin{lemma}\label{Lemma-cdnc}
For $n\geqslant 1$, we have $(cD)^nc=\sum_{p\in\operatorname{OWP}_{n}}w(p)$.
\end{lemma}
\subsection{Main results}
Given $T\in\SYT(n)$. We define $w_i(T)$ to be the number of rows in $T$ with $i$ elements.
Let $\ell(\lambda(T))$ be the number of rows of $T$, where $\lambda$ is the shape of $T$.
Then $\ell(\lambda(T))=\sum_{i=1}^nw_i(T)$ and $n=\sum_{i=1}^niw_i(T)$.
The {\it weight function} of $T$ is defined by
\begin{equation}\label{wT}
w(T)=c^{n+1-\ell(\lambda(T))}\prod_{i=1}^nc_i^{w_i(T)}.
\end{equation}
\renewcommand{\arraystretch}{2}
\begin{center}
\begin{table}[ht!]
  \caption{$\phi^{-1}(T)=\left\{p\in\operatorname{OWP}_{3}| \phi(p)=T\right\}$, where $T\in\SYT(3)$}\label{tab:dummy-1}
  {
\begin{tabular}{ccccc|c}
$T$&&$\stackrel{\phi^{-1}}{\Longrightarrow}$&&$P$&count\\
\hline
    \begin{ytableau}
    *(gray!20) 3\\
    *(gray!20) 2\\
    *(gray!20) 1\\
    *(gray!20) 0 & 1 & 2 &3  \\
    \end{ytableau}
    & $\Longrightarrow$
    &
    \begin{ytableau}
    *(gray!20) 3&\\
    *(gray!20) 2&\\
    *(gray!20) 1&\\
    *(gray!20) 0 & 1 & 2 &3  \\
    \end{ytableau}
    &
    & $\left( \{1, 2, 3 \}, \{ \}, \{ \},\{\} \right)$
    &1\\
\hline
    \begin{ytableau}
    *(gray!20) 3\\
    *(gray!20) 2\\
    *(gray!20) 1 & 2\\
    *(gray!20) 0 & 1 & 3  \\
    \end{ytableau}
    & $\Longrightarrow$
    &
    \begin{ytableau}
    *(gray!20) 3&\\
    *(gray!20) 2&\\
    *(gray!20) 1 & 2\\
    *(gray!20) 0 & 1 & 3  \\
    \end{ytableau}+~~
    \begin{ytableau}
    *(gray!20) 3&\\
    *(gray!20) 2&\\
    *(gray!20) 1 & 2 & 3\\
    *(gray!20) 0 & 1  \\
    \end{ytableau}
    &
    & $( \{1, 3 \}, \{ 2 \}, \{  \},\{\} ), ~ ( \{1 \}, \{ 2, 3 \}, \{  \} ,\{\})$
    &2\\
\hline
    \begin{ytableau}
    *(gray!20) 3\\
    *(gray!20) 2\\
    *(gray!20) 1 & 3\\
    *(gray!20) 0 & 1 & 2  \\
    \end{ytableau}
    & $\Longrightarrow$
    &
    \begin{ytableau}
    *(gray!20) 3&\\
    *(gray!20) 2&\\
    *(gray!20) 1 & 3\\
    *(gray!20) 0 & 1 & 2  \\
    \end{ytableau}+~~
    \begin{ytableau}
    *(gray!20) 3&\\
    *(gray!20) 2 & 3\\
    *(gray!20) 1& \\
    *(gray!20) 0 & 1 & 2  \\
    \end{ytableau}
    &
    & $\left( \{1, 2 \}, \{ 3 \}, \{  \} ,\{\}), ~ ( \{1, 2 \}, \{  \}, \{ 3 \},\{\} \right)$
    &2\\
\hline
    \begin{ytableau}
    *(gray!20) 3\\
    *(gray!20) 2 & 3\\
    *(gray!20) 1 & 2\\
    *(gray!20) 0 & 1\\
    \end{ytableau}
    & $\Longrightarrow$
    &
    \begin{ytableau}
    *(gray!20) 3&\\
    *(gray!20) 2 & 3\\
    *(gray!20) 1 & 2\\
    *(gray!20) 0 & 1\\
    \end{ytableau}
    &
    & $\left( \{1 \}, \{ 2 \}, \{ 3 \},\{\} \right)$
    &1\\
    \hline
\end{tabular}}
\end{table}
\end{center}
Let $\phi$ be the map from $\operatorname{OWP}_{n}$ to $\SYT(n)$, which is described as follows:
\begin{itemize}
  \item [$OS1$:] For $p\in\operatorname{OWP}_{n}$, let $Y$ be the corresponding $\YWCT$.
  Reorder the left-justified rows of $Y$ by their length in decreasing order from the bottom to the top, and delete all empty boxes.
  \item [$OS2$:]Rearrange the entries in each column in ascending order from the bottom to the top.
\end{itemize}
In view of~\eqref{wp} and~\eqref{wT}, we see that for any $p\in\operatorname{OWP}_{n}$, one has $\phi(p)\in\SYT(n)$ and
\begin{equation}\label{wp-SYT}
w(p)=w\left(\phi(p)\right).
\end{equation}

\begin{definition}
Given $T\in\SYT(n)$. Let
$\phi^{-1}(T)=\{p\in\operatorname{OWP}_{n}|~ \phi(p)=T\}$.
We call $\#\phi^{-1}(T)$ the {\it $g$-index} of $T$.
\end{definition}
Clearly, $\#\phi^{-1}(T)=1$ for $T\in \SYT(1)$ or $T\in\SYT(2)$.
The correspondences between $\SYT(3)$ and $\operatorname{OWP}_{3}$ are listed in Table~\ref{tab:dummy-1}.
For $T\in\SYT(n)$, let $T_i$ be the element in $\SYT(i)$ obtained from $T$ by deleting the $n-i$ elements $i+1,i+2,\ldots,n$.
We denote by $\operatorname{col}_k(T_i)$ the size of the $k$-th column of $T_i$.
\begin{theorem}\label{TSYTncol}
For $T\in\SYT(n)$, the $g$-index of $T$ can be computed as follows:
\begin{equation}\label{phiT}
\#\phi^{-1}(T)=\prod_{i=1}^n\sigma_i(T).
\end{equation}
where $\sigma_i(T)$ is defined by
\begin{equation*}
\sigma_i(T)=\left\{
  \begin{array}{ll}
   i-\operatorname{col}_1(T_i)+1, & \hbox{if $i$ is in the first column of $T$;} \\
    \operatorname{col}_{k}(T_i)-\operatorname{col}_{k+1}(T_i)+1, & \hbox{if $i$ is in the $(k+1)$-th column of $T$, where $k\geqslant 1$.}
  \end{array}
\right.
\end{equation*}
We call $\sigma_i(T)$ the $g$-index of the element $i$.
Then we have
\begin{equation}\label{cdnc-Thm}
(cD)^nc=\sum_{T\in\operatorname{SYT}{(n)}}\#\phi^{-1}(T)w(T)=\sum_{T\in\operatorname{SYT}{(n)}}\left(\prod_{i=1}^n\sigma_i(T)c_i^{w_i(T)}\right)c^{n+1-\ell(\lambda(T))}.
\end{equation}
\end{theorem}
\begin{proof}
In order to prove~\eqref{phiT}, we need to count the possible positions of each entry $i$ of $p\in\phi^{-1}(T)$. We distinguish two cases:
\begin{itemize}
  \item [$(i)$] Suppose that $i$ is the $r$-th entry in the first column of $T$. Then $r=\operatorname{col}_1(T_i)$. In $T_i$, the entry $i$ is maximum. By the box sorting algorithm, we see that there are $i-(r-1)$ ways to insert $i$, and each insertion generates an element of $\operatorname{OWP}_i$, see Table~\ref{tab:dummy-1} for illustrations.
  \item [$(ii)$] Suppose that $i$ is the $r$-th entry in the $(k+1)$-th column of $T$, where $k\geqslant 1$. Then $r=\operatorname{col}_{k+1}(T_i)$. In $T_i$, the entry $i$ is maximum.
 By box sorting algorithm, we find that there are $\operatorname{col}_k(T_i)-(r-1)$ ways to insert the entry $i$, and each insertion generates an element of $\operatorname{OWP}_i$.
\end{itemize}
Continuing in this way, we eventually recover all the elements in $\operatorname{OWP}_n$.
By the multiplication principle, we get~\eqref{phiT}.
Combining~\eqref{wp-SYT},~\eqref{phiT} and Lemma~\ref{Lemma-cdnc}, we arrive at~\eqref{cdnc-Thm}, and hence the proof is complete.
\end{proof}

\begin{theorem}\label{AnxSYT}
We have
\begin{equation*}\label{Anx-01}
A_n(x)=\sum_{T\in \SYT(n)}\left(\prod_{i=1}^n \sigma_i(T)\right)x^{\ell(\lambda(T))}.
\end{equation*}
\end{theorem}
\begin{proof}
 Let $G=\{x \rightarrow x, y \rightarrow x\}$.
Note that Proposition~\ref{grammar03} can also be restated as
$$(yD_{G})^{n}(y)|_{y=1}=A_n(x).$$
Taking $c=y$ in~\eqref{cdnc-Thm}, we get $c_i=D_{G}^i(c)=D_{G}^i(y)=x$ for $i\geqslant 1$.
It follows from~\eqref{cdnc-Thm} that
\begin{align*}
A_n(x)&=\sum_{T\in\operatorname{SYT}(n)}\left(\prod_{i=1}^n\sigma_i(T)c_i^{w_i(T)}\right)c^{n+1-\ell(\lambda(T))}{\bigg|}_{c=y=1,~c_i=x}\\
&= \sum_{T\in \SYT(n)} \left(\prod_{i=1}^n \sigma_i(T) \right) ~ x^{\sum_{i=1}^nw_i(T)}\\
&= \sum_{T\in \SYT(n)} \left(\prod_{i=1}^n \sigma_i(T) \right) ~ x^{\ell(\lambda(T))}.
\end{align*}
\end{proof}

\begin{theorem}\label{thm-Cnxy}
We have
$$C_n(x)=\sum_{T\in\operatorname{SYT}{(n)}}\prod_{i=1}^n\sigma_i(T){A_i(x)}^{w_i(T)}.$$
\end{theorem}
\begin{proof}
From Example~\ref{ExG4}, we see that if $G=\{x\rightarrow xy,~y\rightarrow xy\}$, then
\begin{equation}\label{xyCn02}
(yD_{G})^n(y)=y^{2n+1}C_n\left(\frac{x}{y}\right).
\end{equation}
Taking $c=y$, by Proposition~\ref{grammar03}, we see that $c_i|_{y=1}=D_{G}^i(y)|_{y=1}=A_i(x)$ for $i\geqslant 1$.
By~\eqref{cdnc-Thm}, we find the desired formula. This completes the proof.
\end{proof}

An illustration of Theorem~\ref{thm-Cnxy} is given by Table~\ref{tab:dummy-c3xy}.
\renewcommand{\arraystretch}{2}
\begin{center}
\begin{table}[ht!]
  \caption{The computation of $C_3(x)=x+8x^2+6x^3$.}\label{tab:dummy-c3xy}
  {
\begin{tabular}{ccccc|c}
$T$&\quad\quad\quad\quad${\sigma_i(T)},~w_i(T)$&~~~~~~\quad\quad enumerator\\
\hline
    \begin{ytableau}
    *(gray!20) 3\\
    *(gray!20) 2\\
    *(gray!20) 1\\
    *(gray!20) 0 & 1 & 2 &3  \\
    \end{ytableau}
    & $\Longrightarrow$
    &
   $\substack{\sigma_1(T)=\sigma_2(T)=\sigma_3(T)=1\\ w_1(T)=w_2(T)=0,~w_3(T)=1}$
    & &
    &$x+4x^2+x^3$\\
\hline
    \begin{ytableau}
    *(gray!20) 3\\
    *(gray!20) 2\\
    *(gray!20) 1 & 2\\
    *(gray!20) 0 & 1 & 3  \\
    \end{ytableau}
    & $\Longrightarrow$
    &
  $\substack{\sigma_1(T)=\sigma_2(T)=1,~\sigma_3(T)=2\\w_1(T)=w_2(T)=1,~w_3(T)=0}$
    &
    &
    &$2x(x+x^2)$\\
\hline
    \begin{ytableau}
    *(gray!20) 3\\
    *(gray!20) 2\\
    *(gray!20) 1 & 3\\
    *(gray!20) 0 & 1 & 2  \\
    \end{ytableau}
    & $\Longrightarrow$
    &
   $\substack{\sigma_1(T)=\sigma_2(T)=1,~\sigma_3(T)=2\\w_1(T)=w_2(T)=1,~w_3(T)=0}$
    &
    &
    &$2x(x+x^2)$\\
\hline
    \begin{ytableau}
    *(gray!20) 3\\
    *(gray!20) 2 & 3\\
    *(gray!20) 1 & 2\\
    *(gray!20) 0 & 1\\
    \end{ytableau}
    & $\Longrightarrow$
    &
  $\substack{\sigma_1(T)=\sigma_2(T)=\sigma_3(T)=1\\w_1(T)=3,~w_2(T)=w_3(T)=0}$
    &
    &
    &$x^3$\\
    \hline
\end{tabular}}
\end{table}
\end{center}

From Example~\ref{ExG4}, one can easily verify that 
\begin{equation}\label{xyCn}
(xD_{G})^n(x)=y^{2n+1}C_n\left(\frac{x}{y}\right),~\text{where $G=\{x\rightarrow y^2,~y\rightarrow y^2\}$.}
\end{equation}

\begin{theorem}\label{cnx01}
We have
$$C_n(x)=\sum_{T\in\operatorname{SYT}{(n)}}\left(\prod_{i=1}^n\sigma_i(T){i!}^{w_i(T)}\right)x^{n+1-\ell(\lambda(T))}.$$
\end{theorem}
\begin{proof}
Note that
$C_n(x)=(xD_{G})^n(x){|}_{y=1}$, where $G=\{x\rightarrow y^2,~y\rightarrow y^2\}$.
Taking $c=x$, then $c_i=D_{G}^i(c)=D_{G}^i(x)=i!y^{i+1}$ for $i\geqslant 1$.
By~\eqref{cdnc-Thm}, we get
\begin{equation*}
C_n(x)=\sum_{T\in\operatorname{SYT}{(n)}}\left(\prod_{i=1}^n\sigma_i(T)c_i^{w_i(T)}\right)c^{n+1-\ell(\lambda(T))}{\bigg|}_{c=x,~c_{i}=i!y^{i+1},~y=1},
\end{equation*}
which yields the desired result. This completes the proof.
\end{proof}

Let $\SYT(n;k)$ be the subset of $\SYT(n)$ with at most $k$ columns.
\begin{theorem}\label{cnx02}
For the trivariate second-order Eulerian polynomials, we have
$$C_{n+1}(x,y,z)=\sum_{T\in\operatorname{SYT}{(n;3)}}\prod_{i=1}^n\sigma_i(T){c_1}^{w_1(T)}{c_2}^{w_2(T)}6^{w_3(T)}
(xyz)^{n+1-\ell(\lambda(T))},$$
where $c_1=xy+yz+xz$ and $c_2=2x+2y+2z$.
\end{theorem}
\begin{proof}
It follows from~\eqref{DGxyz} that
$(xyzD_{G})^n(xyz)=C_{n+1}(x,y,z)$, where $$G=\{x\rightarrow 1,~y\rightarrow 1, z\rightarrow 1\}.$$
Setting $c=xyz$, we get that $c_1=xy+yz+xz$, $c_2=2x+2y+2z$, $c_3=D_{G}(2x+2y+2z)=6$,
and $c_i=0$ for $i\geqslant 4$.
Substituting $c=xyz,~c_{1}=xy+yz+xz,~c_{2}=2x+2y+2z$, $c_3=6$, and $c_i=0$ for $i\geqslant 4$ into the following expression:
\begin{equation*}
C_{n+1}(x,y,z)=\sum_{T\in\operatorname{SYT}{(n)}}\left(\prod_{i=1}^n\sigma_i(T)c_i^{w_i(T)}\right)
c^{n+1-\ell(\lambda(T))},
\end{equation*}
we obtain the desired formula. This completes the proof.
\end{proof}

We end our paper by giving an expression of the type $B$ Eulerian polynomials.
\begin{theorem}
Let $B_n(x)$ be the type $B$ Eulerian polynomials. We have
$$B_n(x)=\sum_{T\in\operatorname{SYT}{(n)}}\left(\prod_{i=1}^n\sigma_i(T)c_i^{w_i(T)}\right)x^{{\frac{1}{2}}\left(n-\ell(\lambda(T))\right)},$$
where $c_{2i-1}=4^{i-1}(1+x)$ and $c_{2i}=4^i\sqrt{x}$ for $i\geqslant 1$.
\end{theorem}
\begin{proof}
From~\eqref{DG3ab}, we see that
$xB_n(x^2)=(xyD_{G})^n(xy){|}_{y=1}$, where $G=\{x\rightarrow y,~y\rightarrow x\}$.
Taking $c=xy$, we notice that $$c_{2i-1}=D_{G}^{2i-1}(c)=4^{i-1}(x^2+y^2),~c_{2i}=D_{G}^{2i}(c)=4^ixy~{\text{for $n\geqslant 1$}}.$$
By~\eqref{cdnc-Thm}, we find that
\begin{align*}
xB_n(x^2)&=\sum_{T\in\operatorname{SYT}{(n)}}\left(\prod_{i=1}^n\sigma_i(T)c_i^{w_i(T)}\right)c^{n+1-\ell(\lambda(T))}{\bigg|}_{\substack{c_{2i-1}=4^{i-1}(x^2+y^2),~c_{2i}=4^ixy,\\y=1}}\\
&=\sum_{T\in\operatorname{SYT}{(n)}}\left(\prod_{i=1}^n\sigma_i(T)c_i^{w_i(T)}\right)x^{n+1-\ell(\lambda(T))}{\bigg|}_{c_{2i-1}=4^{i-1}(1+x^2),~c_{2i}=4^ix},
\end{align*}
which yields the desired result.
\end{proof}
\bibliographystyle{amsplain}

\begin{thebibliography}{10}
\bibitem{Adin19}
R.M. Adin, S. Elizalde, Y. Roichman, \textit{Cyclic descents for near-hook and two-row shapes}, European J. Combin., \textbf{79} (2019), 152--178.

\bibitem{Bitonti24}
V. Bitonti, B. Deb, A.D. Sokal, \textit{Thron-type continued fractions (T-fractions)
 for some classes of increasing trees}, \arxiv{2412.10214v1}.

\bibitem{Blasiak10}
P. Blasiak, P. Flajolet, \textit{Combinatorial models of creation-annihilation}, S\'em. Lothar. Combin., \textbf{65} (2010/12), Art. B65c, 78 pp.


\bibitem{Bona08}
M. B\'ona, \textit{Real zeros and normal distribution for statistics on Stirling permutations defined by Gessel and Stanley}, SIAM J. Discrete Math., \textbf{23} (2008/09), 401--406.




\bibitem{Brenti94}
F. Brenti, \textit{$q$-Eulerian polynomials arising from Coxeter groups}, European J. Combin., \textbf{15} (1994), 417--441.

\bibitem{Brenti00}
F. Brenti, \textit{A class of $q$-symmetric functions arising from plethysm},
J. Combin. Theory Ser. A, \textbf{91} (2000), 137--170.

\bibitem{Briand20}
E. Briand, S. Lopes, M. Rosas, \textit{Normally ordered forms of powers of differential operators and their combinatorics}, J. Pure Appl. Algebra, \textbf{224}(8) (2020), 106312.


\bibitem{Brualdi20}
R.A. Brualdi, \textit{Stirling pairs of permutations}, Graphs Comb., \textbf{36} (2020), 1145--1162.



\bibitem{Buckholtz}
J.D. Buckholtz, \textit{Concerning an approximation of Copson}, Proc. Amer. Math. Soc., 14 (1963), 564--568.


\bibitem{Cahill02}
N.D. Cahill, J.R. D'Errico, D.A. Narayan, J.Y. Narayan, \textit{Fibonacci determinants}, College Math. J., \textbf{33} (2002), 221--225.


\bibitem{Carlitz65}
L. Carlitz, \textit{The coefficients in an asymptotic expansion}, Proc. Amer. Math. Soc., \textbf{16} (1965) 248--252.

\bibitem{Chen93}
W.Y.C. Chen, \textit{Context-free grammars, differential operators and formal
power series}, Theoret. Comput. Sci., \textbf{117} (1993), 113--129.



\bibitem{Chen17}
W.Y.C. Chen, A.M. Fu, \textit{Context-free grammars for permutations and increasing trees}, Adv. in Appl. Math., \textbf{82} (2017), 58--82.

\bibitem{Chen22}
W.Y.C. Chen, A.M. Fu, \textit{A context-free grammar for the $e$-positivity of the trivariate second-order Eulerian polynomials},
Discrete Math., \textbf{345}(1) (2022), 112661.
%
\bibitem{Chen23}
W.Y.C. Chen, A.M. Fu, S.H.F. Yan, \textit{The Gessel correspondence and the partial $\gamma$-positivity of the Eulerian polynomials on multiset Stirling permutations},
European J. Combin., \textbf{109} (2023), 103655.


\bibitem{Chen2102}
W.Y.C. Chen, R.X.J. Hao and H.R.L.Yang, \textit{Context-free grammars and multivariate stable polynomials over Stirling permutations}, In: V. Pillwein and
C. Schneider (eds.), Algorithmic Combinatorics: Enumerative Combinatorics, Special Functions and Computer Algebra, pp. 109--135, Springer, 2021.



\bibitem{Chow09}
C.-O. Chow, \textit{On derangement polynomials of type $B$}, II, J. Combin. Theory Ser. A, \textbf{116} (2009),
816--830.

\bibitem{Comtet73}
L. Comtet, \textit{Une formule explicite pour les puissances successives de l'operateur de d\'{e}rivations de Lie},
C. R. Hebd. Seances Acad. Sci., \textbf{276} (1973), 165--168.

\bibitem{Dumont80}
D. Dumont, \textit{Une g\'en\'eralisation trivari\'ee sym\'etrique des nombres eul\'eriens}, J. Combin. Theory Ser. A, \textbf{28} (1980), 307--320.
%
\bibitem{Dumont96}
D. Dumont, \textit{Grammaires de William Chen et d\'erivations dans les arbres et
arborescences}, S\'em. Lothar. Combin., \textbf{37}, Art. B37a (1996), 1--21.
%

\bibitem{Elizalde}
S. Elizalde, \textit{Descents on quasi-Stirling permutations}, J. Combin. Theory Ser. A, \textbf{180} (2021), 105429. 


%




\bibitem{Gessel78}
I. Gessel, R.P. Stanley, \textit{Stirling polynomials}, J. Combin. Theory Ser. A, \textbf{24} (1978), 25--33.

\bibitem{Gessel20}
I. Gessel, \textit{A note on Stirling permutations}, \arxiv{2005.04133}.

\bibitem{Haglund12}
J. Haglund and M. Visontai, \textit{Stable multivariate Eulerian polynomials and
generalized Stirling permutations}, European J. Combin., \textbf{33} (2012), 477--487.

\bibitem{Han23}
G.-N. Han, S.-M. Ma, \textit{Eulerian polynomials and the $g$-indices of Young tableaux}, Proc. Amer. Math. Soc., \textbf{152} (2024), 1437--1449.


\bibitem{Zeng2020}
B. Han, J. Mao, J. Zeng, \textit{Eulerian polynomials and excedance statistics}, Adv. in Appl. Math., \textbf{121} (2020), 102092.


\bibitem{Huang20}
B. Huang, \textit{Cyclic descents for general skew tableaux}, J. Comb. Theory, Ser. A, \textbf{169} (2020), 105120.



\bibitem{Hwang20}
H.-K. Hwang, H.-H. Chern, G.-H. Duh, \textit{An asymptotic distribution theory for Eulerian
recurrences with applications}, Adv. in Appl. Math., \textbf{112} (2020), 101960.
%
\bibitem{Janson11}
S. Janson, M. Kuba, A. Panholzer, \textit{Generalized Stirling permutations, families of increasing trees and urn models}, J. Combin.
Theory Ser. A, \textbf{118}(1) (2011), 94--114.
%

\bibitem{Kilic17}
E. K{\i}l{\i}\c{c}, Talha Ar{\i}kan, \textit{Evaluation of Hessenberg determinants via generating function}, Filomat, 31(15) (2017), 4945--4962.


\bibitem{Kuba21}
M. Kuba, A.L. Varvak, \textit{On path diagrams and Stirling permutations}, S\'em. Lothar. Combin., \textbf{82} (2021), Article B82c.



\bibitem{Lin21}
Z. Lin, J. Ma, P.B. Zhang, \textit{Statistics on multipermutations and partial $\gamma$-positivity}, J. Comb. Theory, Ser. A,
\textbf{183} (2021), 105488.


\bibitem{Liu07}
L.L. Liu, Y. Wang, \textit{A unified approach to polynomial sequences with only real zeros},  Adv. in Appl. Math., \textbf{38} (2007), 542--560.


\bibitem{Liu23}
S.H. Liu, \textit{The Haglund-Remmel-Wilson identity for $k$-Stirling permutations}, European J. Combin., \textbf{110} (2023), 103676.

%


\bibitem{MaYeh17}
S.-M. Ma, Y.-N. Yeh, \textit{Eulerian polynomials, Stirling permutations
of the second kind and perfect matchings}, Electron. J. Combin., \textbf{24}(4) (2017), \#P4.27.

%


\bibitem{Ma19}
S.-M. Ma, J. Ma, Y.-N. Yeh, \textit{$\gamma$-positivity and partial $\gamma$-positivity of descent-type polynomials}, J. Combin. Theory Ser. A, \textbf{167} (2019), 257--293.


\bibitem{Ma23}
S.-M. Ma, H. Qi, J. Yeh, Y.-N. Yeh, \textit{Stirling permutation codes}, J. Combin. Theory Ser. A, \textbf{199} (2023), 105777.

\bibitem{Ma24}
S.-M. Ma, J. Ma, J. Yeh, Y.-N. Yeh, \textit{Excedance-type polynomials, gamma-positivity and alternatingly increasing property}, European J. Combin., \textbf{118} (2024), 103869.




\bibitem{Rota92}
M. Kac, G.-C. Rota, J. Schwartz, \textit{Discrete thoughts}, Birkh\"auser, Boston, 1992.




\bibitem{Remmel14}
J.B. Remmel, A.T. Wilson, \textit{Block patterns in Stirling permutations}, \arxiv{1402.3358}.


%

\bibitem{Schork21}
M. Schork, \textit{Recent developments in combinatorial aspects of normal ordering}, Enumer. Combin. Appl., \textbf{1} (2021), Article S2S2.


\bibitem{Sloane}
N.J.A. Sloane, \textit{The On-Line Encyclopedia of Integer Sequences},
published electronically at https://oeis.org, 2010.

%




%
%
\end{thebibliography}

\end{document}